\documentclass[12pt,a4paper,reqno]{amsart}
\usepackage{amsfonts,amssymb,a4wide,hyperref}

\newcommand{\N}{\mathbb{N}}              
\newcommand{\Z}{\mathbb{Z}} 
\newcommand{\R}{\mathbb{R}}       
\newcommand{\Rn}{{\R}^n}

\newcommand{\cs}{{\mathcal S}}
\newcommand{\cf}{{\mathcal F}}
\newcommand{\cfi}{{\mathcal F}^{-1}}

\newcommand{\supp}{\operatorname{supp}}

\newtheorem{thm}{Theorem}
\newtheorem{prop}{Proposition}
\newtheorem{lem}{Lemma}
\newtheorem{cor}{Corollary}

\newtheorem*{theorem*}{{\bf Theorem\/}}

\theoremstyle{definition}
\newtheorem{defn}{Definition}
 
\theoremstyle{remark}
\newtheorem{rem}{Remark}
\title[Local Means and Mixed Norms]{Characterisation by Local Means of 
Anisotropic Lizorkin--Triebel Spaces with Mixed Norms}
\author[J.~Johnsen, S.~Munch Hansen, W.~Sickel]{J.~Johnsen, S.~Munch~Hansen, W.~Sickel}
\thanks{J.~Johnsen supported by the Danish Council for Independent Research, Natural Sciences (Grant No. 09-065927 and 11-106598)
\\[8\jot]{\tt Appeared in Journal for Analysis and its Applications (ZAA), {\bf 32} (2013), 257--277.}} 

\address{Department of Mathematical Sciences, Aalborg University, Fredrik Bajers Vej 7G,
  DK--9220 Aalborg {\O}st, Denmark} 
  \email{jjohnsen@math.aau.dk}
\address{Department of Mathematical Sciences, Aalborg University, Fredrik Bajers Vej 7G,
  DK--9220 Aalborg {\O}st, Denmark} 
  \email{sabrina@math.aau.dk}
\address{Institute of Mathematics,Friedrich-Schiller-University Jena, 
Ernst Abbe Platz 1--2, D-07743 Jena, Germany}
  \email{Winfried.Sickel@uni-jena.de}

\begin{document}
 \begin{abstract}
This is a contribution to the theory of Lizorkin--Triebel spaces having mixed Lebesgue norms
and quasi-homogeneous smoothness. 
We discuss their characterisation in terms of general quasi-norms based on convolutions.
In particular, this covers the case of local means, in Triebel's
terminology. The main step is an extension of some crucial inequalities due to Rychkov to the
case with mixed norms.
 \end{abstract}
\enlargethispage{3\baselineskip}
\maketitle
 
\section{Introduction}
This paper is devoted to a study of anisotropic Lizorkin--Triebel spaces 
$F^{s,\vec a}_{\vec p,q}(\Rn)$ with mixed norms, which has grown out of work of the first and third
author \cite{JoSi07,JoSi08}.

First Sobolev embeddings and completeness of the scale $F^{s,\vec a}_{\vec p,q}(\Rn)$ 
were covered in \cite{JoSi07}.
As the foundation for this, the Nikol'ski\u\i--Plancherel--Polya inequality for
sequences of functions in the mixed-norm space $L_{\vec p}(\Rn)$ 
was established in \cite{JoSi07} with fairly elementary proofs. Then a detailed 
trace theory for hyperplanes in $\Rn$ was worked out in \cite{JoSi08}, 
e.g.\ with the novelty that the well-known borderline $s=1/p$ has to be shifted upwards in some cases,
because of the mixed norms. 

In the present paper we obtain some general characterisations of the space 
$F^{s,\vec a}_{\vec p,q}(\Rn)$, that may be specialised to kernels of local means. 
We have at least two motivations for this. One is that local means have
emerged in the last decade as the natural foundation for a discussion of wavelet bases for Sobolev
spaces and their generalisations to the Besov and Lizorkin--Triebel scales; cf.\ works of
Triebel~\cite[Thm.~1.20]{Tri08} and e.g.\ Vybiral~\cite[Thm.~2.12]{Vyb06}, 
Hansen~\cite[Thm.~4.3.1]{Han10}.

Secondly, local means will be crucial for the entire strategy in our forthcoming paper
\cite{JoMHSi2}, in which we establish invariance of $F^{s,\vec a}_{\vec p,q}$ under diffeomorphisms
in order to carry over trace results from \cite{JoSi08} to spaces over smooth
domains. More precisely, because of the anisotropic structure of the 
$F^{s,\vec a}_{\vec p,q}$-spaces, we consider them over smooth cylindrical domains in Euclidean
space in \cite{JoMHSi2} 
and develop results for traces on the flat and curved parts of the boundary of the cylinder
in \cite{JoMHSi3}.

To elucidate the importance of the results here and in \cite{JoMHSi2,JoMHSi3}, we recall that
$F^{s,\vec a}_{\vec p,q}$-spaces have applications to parabolic differential equations
with initial and boundary value conditions: when solutions are
sought in a \emph{mixed-norm} Lebesgue space $L_{\vec p}$ 
(e.g.~to allow for different properties in the space and time directions), then
$F^{s,\vec a}_{\vec p,q}$-spaces are in general \emph{inevitable} for a correct
description of non-trivial data on the \emph{curved} boundary.

This conclusion was obtained in works of Weidemaier \cite{Wei98,Wei02,Wei05}, 
who treated several special cases; the reader may consult the introduction of
\cite{JoSi08} for details.

To give a brief review of the present results, we recall that the norm 
$\|\cdot| F^{s,\vec a}_{\vec p,q}\|$ of $F^{s,\vec a}_{\vec p,q}(\Rn)$ is defined 
in a well-known Fourier-analytic way by splitting the frequency space by means of a
Littlewood--Paley partition of unity. 

But to have `complete' freedom, it is natural first of all to work with
convolutions $\psi_j*f$ defined from more arbitrary sequences $(\psi_j)_{j\in\N_0}$ of
Schwartz functions with dilations $\psi_j=2^{j|\vec a|}\psi(2^{j\vec a}\cdot)$ for $j\ge1$. This
requires both the Tauberian conditions that 
$\widehat\psi_0(\xi),\,\widehat\psi(\xi)$ have no zeroes for
$|\xi|_{\vec a}<2\varepsilon$ and $\frac{\varepsilon}{2}  <|\xi|_{\vec a}<\varepsilon$, respectively;
and the moment condition that $D^\alpha\widehat\psi(0)=0$ for $|\alpha|\le M_\psi$.

Secondly, one may work with anisotropic Peetre--Fefferman--Stein maximal functions
$\psi_{j,\vec a}^*f$, and with these our main result can be formulated as follows:

\begin{theorem*}
  If $s<(M_\psi+1)\min(a_1,\dots,a_n)$ and $0<p_j<\infty$, $0<q\le\infty$, the following
  quasi-norms are equivalent on the space of temperate distributions:
  \begin{equation}
    \| f | F^{s,\vec a}_{\vec p,q}\|,
\qquad
    \| \{ 2^{sj}\psi_j*f\}_{j=0}^{\infty} | L_{\vec p}(\ell_q) \|,
\qquad
    \| \{ 2^{sj}\psi_{j,\vec a}^*f\}_{j=0}^{\infty} | L_{\vec p}(\ell_q) \|.
  \end{equation}
Thus $f\in F^{s,\vec a}_{\vec p,q}({\mathbb{R}}^n)$ if and only if one
(hence all) of these expressions are finite.
\end{theorem*}

In the isotropic case, i.e.\ when $\vec a=(1,\dots,1)$ and unmixed $L_p$-norms are used,
the theorem has been known since the important work of Rychkov~\cite{Ry99BPT},
albeit in another formulation. In our generalisation we follow Rychkov's proof
strategy closely, but with some corrections; cf.\ Remark~\ref{Rych0-rem}~below. 

Another particular case is when the functions $\psi_0$ and $\psi$ have compact support, in
which case the convolutions may be interpreted as local means, as observed by
Triebel~\cite{T3}. Thus we develop the mentioned characterisations by local means for the
anisotropic $F^{s,\vec a}_{\vec p,q}$-spaces in Theorem~\ref{local} below, 
and as far as we know, already this part
of their theory is a novelty. As indicated above, it will enter directly into the proofs of our
paper \cite{JoMHSi2}. 

However, it deserves to be mentioned that the arguments in \cite{JoMHSi2} also rely on a stronger
estimate than the inequalities underlying the above theorem. In fact we need to consider parameter
dependent functions $\psi_\theta$, $\theta\in\Theta$ (an index set), that satisfy moment
conditions in a uniform way. Theorem~\ref{maxmod} below gives the precise details and our
estimate of
\begin{equation}
  \|\, \{2^{sj}\sup_{\theta\in\Theta}\psi_{\theta,j,\vec a}^*f\}_{j=0}^\infty \,| L_{\vec
    p}(\ell_q)\|.
\label{theta-eq}
\end{equation}
Similar quasi-norms were introduced by Triebel in the proof of \cite[Prop.~4.3.2]{T3}
for the purpose of showing diffeomorphism invariance of the isotropic scale
$F^s_{p,q}(\Rn)$. However, he only claimed the equivalence of the quasi-norms for
$f$ belonging a priori to $F^s_{p,q}$ and details of proof were not given. Since our estimate of
\eqref{theta-eq} is valid for arbitrary  distributions $f\in{\mathcal{S}}'$, it should be well motivated that we develop
this important tool with a full explanation here.

\begin{rem} \label{Rych0-rem}
The fact that the arguments in \cite{Ry99BPT} are incomplete was observed
in the Ph.D.~thesis of M.~Hansen~\cite[Rem.~3.2.4]{Han10}, where it was exemplified that 
in general a certain $O$-condition is unfulfilled; cf.\ Remark~\ref{Rych-rem} below.
Another flaw is pointed out here in Remark~\ref{Rych1-rem}.
However, to obtain the full generality with
arbitrary temperate distributions in Proposition~\ref{prop:pointwiseEstimate} below, 
we have preferred to reinforce the original proofs of Rychkov. Hence we have 
found it best to aim at a self-contained exposition in this paper.
\end{rem}

\subsection*{Contents}
The paper is organized as follows.
Section~\ref{prel-sect} reviews our notation and gives a discussion of the anisotropic spaces 
of Lizorkin--Triebel type with a mixed norm.  Section~\ref{max-sect} presents some maximal
inequalities for mixed Lebesgue norms. Quasi-norms defined from general systems of Schwartz
functions subjected to moment and Tauberian conditions are estimated in
Section~\ref{Rych-sect}, following works of Rychkov.
In Section~\ref{localMeans} 
these spaces are characterised by such general norms, and by local means.

\section{Preliminaries} \label{prel-sect}
\subsection{Notation}
Vectors $\vec{p}=(p_1,\ldots,p_n)$ with $p_i\in\, ]0,\infty]$ for $i=1,\ldots,n$ are written
$0<\vec{p}\leq \infty$, as throughout inequalities for vectors are understood componentwise;
as are functions, e.g.~$\vec{p}\,!=p_1!\cdots p_n!$. 

By $L_{\vec{p}}({\mathbb{R}}^n)$ we denote the set of all functions $u: {\mathbb{R}}^n \to {\mathbb{C}}$ that are Lebesgue
measurable and such that 
\begin{equation}
\|\, u \, |\, L_{\vec{p}}(\R^n)\| := 
\Big(\int_{\mathbb{R}} \Big( \ldots 
\Big( \int_{\mathbb{R}} |u(x_1,\ldots ,x_n)|^{p_1} dx_1
\Big)^{\frac{p_2}{p_1}}  
\ldots  
\Big)^{\frac{p_n}{p_{n-1}}} dx_n \Big)^{\frac1{p_n}} <\infty,
\end{equation}
with the modification of using the essential supremum over $x_j$ in case $p_j=\infty$.
Equipped with this quasi-norm, $L_{\vec{p}}(\R^n)$ is a 
quasi-Banach space; it is normed if $\min (p_1, \ldots , p_n) \ge 1$.

Furthermore, for $0 < q \le \infty$ we shall use the notation
$L_{\vec{p}} (\ell_q)(\Rn)$ for the space of sequences 
$(u_k)_{k\in\N_0}=\{u_k\}_{k=0}^\infty$ of Lebesgue measurable
functions fulfilling
\begin{equation}
\| \, \{u_k\}_{k=0}^\infty \, |L_{\vec{p}} (\ell_q)(\R^n)\| :=
\Big\| \, \Big(\sum_{k=0}^\infty |u_k|^q
\Big)^{1/q}\Big|L_{\vec{p}} (\R^n)\Big\| < \infty, 
\end{equation}
with supremum over $k$ in case $q=\infty$.
For brevity, we write 
$\| \, u_k \, |L_{\vec{p}} (\ell_q)\| $ instead of $
\| \, \{u_k\}_{k=0}^\infty \, |L_{\vec{p}} (\ell_q)({\mathbb{R}}^n)\|$;
as customary for $\vec{p}=(p,\ldots,p)$, we simplify
$L_{\vec{p}}$ to $L_p$ etc.
If $\max (p_1, \ldots, p_n, q)<\infty$,  sequences of 
$C_0^\infty$-functions  are dense in $L_{\vec{p}} (\ell_q)$. 

The Schwartz space $\cs(\Rn)$ consists of all smooth, rapidly decreasing functions;
it is equipped with the family of seminorms, using $\langle x\rangle^2 := 1+|x|^2$,
\begin{equation}\label{eq:seminormOneParam}
  p_M(\varphi) := \sup \big\{ \langle x\rangle^M |D^\alpha \varphi(x)|\, \big|\, 
  x\in\Rn, |\alpha|\leq M\big\},\quad M\in\N_0,
\end{equation}
whereby $D^\alpha := (-\mathrm{i}\partial_{x_1})^{\alpha_1}\cdots(-\mathrm{i}\partial_{x_n})^{\alpha_n}$ 
for each multi-index $\alpha \in\N^n_0$; or with
\begin{equation}\label{eq:seminormsSchwartz}
p_{\alpha,\beta}(\varphi) := \sup_{x\in\Rn} |x^\alpha D^\beta \varphi(x)|,\quad \alpha,\beta\in \N_0^n.
\end{equation}
The Fourier transformation $\cf g(\xi)=\widehat{g}(\xi) = \int_{\Rn} e^{-\mathrm{i} x\cdot \xi} g(x)\, dx$ for $g\in\cs(\Rn)$
extends by duality to the dual space $\cs'(\Rn)$ of temperate distributions.

Throughout generic constants will mainly be denoted by $c$ or $C$, and 
in case their dependence on certain parameters is relevant this will be explicitly stated.

\subsection{Lizorkin--Triebel spaces with a mixed norm}
  \label{mixd}
As a motivation for the general mixed-norm Lizorkin--Triebel
spaces $F^{s,\vec a}_{\vec p,q}(\Rn)$, we first mention that for $1<\vec{p}<\infty$ 
a temperate distribution $u$ belongs to a class $F^{s,\vec{a}}_{\vec{p},2} (\Rn)$ having
natural numbers $m_j:=s/a_j$ for each $j=1,\dots,n$ if and only if $u$ belongs to
the mixed-norm Sobolev space $W^{\vec{m},\vec a}_{\vec p}({\mathbb{R}}^n)$, $\vec{m}=(m_1,\ldots,m_n)$,
defined by
\begin{equation}
\label{eq:mixedNormSobolev}
\| \, u \, |L_{\vec{p}}(\Rn)\| +\sum_{i=1}^n \Big\| \, 
\frac{\partial^{m_i} u}{\partial x_i^{m_i}}\, \Big| L_{\vec{p}}(\Rn)\Big\|
< \infty.
\end{equation}
This expression defines the norm on $W^{\vec{m},\vec a}_{\vec p}$, which is 
equivalent to that on $F^{s,\vec{a}}_{\vec{p},2}$.

More generally, mixed-norm Lizorkin--Triebel spaces generalise the fractional Sobolev
(Bessel potential) spaces 
$H^{s,\vec a}_{\vec p}(\Rn)$, since for $1 < \vec{p} < \infty$ and $s \in \R$, 
\begin{equation}
  u\in H^{s,\vec a}_{\vec p}(\Rn)  \iff 
  u\in F^{s,\vec{a}}_{\vec{p},2} (\Rn).
\label{HF-id}
\end{equation}
Here the norms are also equivalent; the former is given by 
$\|\,{\mathcal{F}}^{-1}(\langle\xi\rangle^{-s}_{\vec a}\widehat u(\xi))\,|L_{\vec p}\|$, whereby 
$\langle\xi\rangle_{\vec a}$ is an anisotropic version of $\langle \xi\rangle$ compatible with
$\vec a$; cf.~the following.

To account for the Fourier-analytic definition of $F^{s,\vec a}_{\vec p,q}(\Rn)$,
we first recall the anisotropic structure used for derivatives.
Each coordinate $x_j$ in $\Rn$ is given a weight $a_j\geq 1$, collected in $\vec
a=(a_1,\dots,a_n)$.
Based on the quasi-homogeneous dilation 
$t^{\vec a}x:=(t^{a_1}x_1,\dots,t^{a_n}x_n)$ for $t\ge0$, and 
$t^{s\vec a}x:=(t^s)^{\vec a}x$ for $s\in \R$, in particular 
$t^{-\vec a}x=(t^{-1})^{\vec a}x$, the anisotropic distance
function $|x|_{\vec a}$ is introduced for $x\neq 0$ as the unique $t>0$ such
that $t^{-\vec a}x\in S^{n-1}$ (with $|0|_{\vec a}=0$); i.e.\
\begin{equation}
  \frac{x_1^2}{t^{2a_1}}+\dots+  \frac{x_n^2}{t^{2a_n}}=1.
  \label{tax-eq}
\end{equation}
For the reader's convenience 
we recall that $|\cdot|_{\vec a}$ 
is $C^\infty$ on $\Rn\setminus\{0\}$ by the Implicit
Function Theorem. The formula $|t^{\vec a}x|_{\vec a}=t|x|_{\vec a}$ is seen
directly, and this implies the triangle inequality,
\begin{equation}
  |x+y|_{\vec a}\le |x|_{\vec a}+|y|_{\vec a}.
  \label{atriangle-ineq}
\end{equation} 
The relation to e.g.~the Euclidean norm $|x|$ can be deduced from
\begin{equation}
  \max(|x_1|^{1/a_1},\dots,|x_n|^{1/a_n})
  \le |x|_{\vec a}\le |x_1|^{1/a_1}+\dots+|x_n|^{1/a_n}.
  \label{ax-ineq}
\end{equation}
For the above-mentioned weight function, one can e.g.~let
$\langle\xi\rangle_{\vec a}=|(\xi,1)|_{(\vec a,1)}$, using the
anisotropic distance given by $(\vec a,1)$ on $\R^{n+1}$; analogously to 
$\langle\xi\rangle$ in the isotropic case.

We pick for convenience a fixed Littlewood-Paley decomposition, 
written $1\!=\sum_{j=0}^\infty \Phi_j(\xi)$, in the anisotropic setting as follows:
Let $\psi \in C^\infty_0$ be a function such that
$0 \le \psi (\xi) \le 1$ for all $\xi$,
$\psi (\xi) =1$ if $|\xi |_a \le 1$, and $\psi (\xi) =0 $ if 
$|\xi|_a \ge 3/2$.
Then we set $\Phi=\psi-\psi(2^{\vec a}\cdot)$ and define
\begin{equation}\label{unity}
  \Phi_0 (\xi) = \psi (\xi), \qquad \Phi_j (\xi) =
  \Phi (2^{-j\vec{a}}\xi), \quad j=1,2,\ldots 
\end{equation}

\begin{defn} \label{F-defn}
The Lizorkin--Triebel space $F^{s,\vec a}_{\vec p,q}(\Rn)$, where $0<\vec{p} < \infty$ 
is a vector of integral exponents, $s\in \R$ a smoothness index, and 
$0 < q \le \infty$ a sum exponent, is the space of all $u\in\cs'(\R^n)$ such that
\begin{equation}
\| \, u \, |\, F^{s,\vec a}_{\vec p,q}(\Rn)\| := 
\Big\| \Big(\sum_{j=0}^\infty 2^{jsq}\, \left|\cfi \left(\Phi_j (\xi) \, 
\cf u(\xi)\right) ( \cdot )\right|^q \Big)^{1/q} \Big|\, L_{\vec{p}}(\Rn)
\Big\| < \infty  .
\end{equation}
\end{defn}

For simplicity, we omit $\vec a$ when $\vec a = (1,\ldots,1)$ and shall often set
\begin{equation}
u_j (x) := \cfi \left(\Phi_j (\xi) \, \cf u(\xi)\right) (x), \quad x\in \Rn\, , \quad 
j\in \N_0\, .
\end{equation}

Occasionally, we need to consider Besov spaces, which are defined very similarly:

\begin{defn}
For $0<\vec{p}\leq \infty$, $0<q\leq \infty$ and $s\in\R$ 
the Besov space $B^{s,\vec a}_{\vec p,q}(\R^n)$ consists of all $u\in\cs'(\R^n)$ such that
\begin{equation}
  \|\, u\, |\, B^{s,\vec a}_{\vec p,q} \| :=
  \Big( \sum_{j=0}^\infty 2^{jsq} \|\, u_j\, | L_{\vec{p}}(\R^n) \|^q\Big)^{1/q} <\infty.
\end{equation}
\end{defn}

\begin{rem}
The Lizorkin--Triebel spaces $F^{s,\vec a}_{\vec p,q}$ have a long history, as  
they give back e.g.\ the mixed-norm Sobolev spaces $W^{\vec m}_{\vec p}$, cf.\ \eqref{eq:mixedNormSobolev}. 
Anisotropic Sobolev (Bessel potential) spaces    
$H^{s,\vec a}_p$ with $1<p<\infty$ (partly for $s>0$) 
have been investigated in the monographs of Nikol'ski\u\i~\cite{Nik75} and  
Besov, Il'in and Nikol'ski\u\i~\cite{BIN78};
here the point of departure was a definition based on derivatives and differences.
In the second edition  \cite{BIN96} also Lizorkin--Triebel spaces with mixed norms were treated in
Ch.~6.29--30. 
For characterisation of $F^{s,\vec a}_{p,q}$ by differences we refer also to
Yamazaki~\cite[Thm.~4.1]{Y2} and Seeger~\cite{Se89}.

The $F^{s,\vec a}_{\vec p,q}$-spaces were considered for $n=2$ by
Schmeisser and Triebel \cite{ScTr87}, who used the Fourier-analytic characterisation, 
which we prefer for its efficacy what concerns application of powerful
tools from Fourier analysis and distribution theory. 
(The definition of the anisotropy in terms of $|\cdot|_{\vec a}$ is a well-known procedure going
back to the 1960's; historical remarks and some basic properties of $|\cdot|_{\vec a}$
can be found in e.g.\ \cite{Y1}.)
\end{rem}

For later use we recall some properties of these classes. 
First standard arguments, cf.~\cite{JoSi07,JoSi08}, yield the following:

\begin{lem}
\label{lem:contEmbed} 
Each $F^{s,\vec a}_{\vec p,q} (\Rn)$ is a quasi-Banach space, which is normed
if both $\vec p\ge1$ and $q \ge 1$. More precisely,
for $u,v\in F^{s,\vec a}_{\vec p,q}$ and  $d:=\min(1,p_1, \, \ldots, p_n,q)$,
\begin{equation}
  \| \, u+v \, |F^{s,\vec a}_{\vec p,q}\|^d \le 
  \| \, u \, |F^{s,\vec a}_{\vec p,q} \|^d \, + \,  \| \, v \, |F^{s,\vec a}_{\vec p,q} \|^d.
\end{equation}
Furthermore, there are continuous embeddings
\begin{equation}
  \cs (\Rn) \hookrightarrow F^{s,\vec a}_{\vec p,q} (\Rn) \hookrightarrow \cs' (\Rn),  
\end{equation}
where $\cs$ is dense in $F^{s,\vec a}_{\vec p,q}$ for $q<\infty$.
Also, the classes $F^{s,\vec a}_{\vec p,q} $ do not depend on the chosen 
anisotropic decomposition of unity (up to equivalent quasi-norms).
\end{lem}

\begin{lem}[\cite{JoSi08}]\label{normalisation}
For $\lambda >0$ so large that $\lambda\vec{a}\geq 1$, the space
$F^{s,\vec a}_{\vec p,q} $ coincides with $F^{\lambda s, \lambda
  \vec{a}}_{\vec{p},q}$ and the corresponding quasi-norms are
equivalent. 
\end{lem}

The lemma suggests to introduce a normalisation
for the vector $\vec{a}$, and often one has fixed the value of 
$|\vec{a}|$ in the literature. In this paper we just adopt the
flexible framework with $\vec{a}\geq 1$, though.

\begin{rem}
In Lemma~\ref{normalisation} the inequalities $\vec{a}\geq 1$ and $\lambda \vec a\ge 1$ are redundant.
In fact one can define $F^{s,\vec a}_{\vec p,q}$ for
arbitrary $\vec a>0$, as in \cite{JoSi08}. This gives another set-up
on $\Rn$, where \eqref{atriangle-ineq}, and hence \eqref{ax-ineq}, has to be changed, for then
\begin{equation}
\label{eq:trianglePower}
  |x+y|_{\vec a}^d\le |x|_{\vec a}^d+|y|_{\vec a}^d,\quad
  d=\min(1,a_1,\dots,a_n).
\end{equation}
The basic results on the $F^{s,\vec a}_{\vec p,q}$-scale can then be
derived similarly for $\vec a>0$; only a few constants need to be 
slightly changed because of \eqref{eq:trianglePower}. Thus one finds
e.g.\ Lemma~\ref{normalisation} for all $\lambda>0$,  cf.\ the end of Section~3 in \cite{JoSi08} 
(the details in \cite[Sec.~3]{JoSi08} only cover $\vec a\ge 1$, but are 
extended to all $\vec a>0$ as just indicated; in fact $\rho(x,y)=|x-y|_{\vec a}$ is then a
quasi-distance, a framework widely used by e.g.\ Stein~\cite{Ste93}). 
However, in view of this lemma, it is simplest henceforth just to
assume that $F^{s,\vec a}_{\vec p,q}$ is defined in terms of an anisotropy
$\vec a\ge 1$; which has been done throughout in the present paper.
\end{rem}

\subsection{Summation lemmas}
For later reference we give two minor results.

\begin{lem} \label{sum1-lem}
  When $(g_j)_{j\in\N_0}$ is a sequence of nonnegative measurable
  functions on $\Rn$ and $\delta>0$, then $G_j(x):=\sum_{k=0}^\infty
  2^{-\delta|j-k|}g_j(x)$ fulfils for $0<\vec p<\infty$,
  $0<q\le\infty$ that
  \begin{equation}
    \|\,G_j\, | L_{\vec p}(\ell_q)\| \le C_{\delta,q}     \|\,g_j\, | L_{\vec p}(\ell_q)\|,
  \end{equation}
whereby the constant is 
$C_{\delta,q}=(\sum_{k\in\Z} 2^{-\delta|k|\tilde q})^{1/\tilde q}$ for $\tilde q=\min(1,q)$.
\end{lem}

Like for the unmixed case in \cite[Lem.~2]{Ry99BPT}, the above lemma is obtained by pointwise
application of Minkowski's inequality to a convolution in $\ell_q(\Z)$.

\begin{lem} \label{sum2-lem}
  Let $(b_j)_{j\in\N_0}$ and $(d_j)_{j\in\N_0}$ be two sequences in $[0,\infty]$ and
  $0<r\le1$. If for some $j_0\ge 0$ there exists real numbers $C, N_0>0$ such that
  \begin{equation}
    d_j\le C2^{jN_0}\quad\text{for}\quad j\ge j_0,
\label{Ry0-ineq}
  \end{equation}
and if for every $N>0$ there exists a real number $C_N$ such that
  \begin{equation}
    d_j\le C_N\sum_{k=j}^\infty 2^{(j-k)N}b_k d_k^{1-r}, \quad\text{for}\quad j\ge j_0,
  \label{Ry1-ineq}
  \end{equation}
then the same constants $C_N$, $N>0$, fulfil that
  \begin{equation}
    d_j^r\le C_N\sum_{k=j}^\infty 2^{(j-k)Nr}b_k, 
  \quad\text{for}\quad j\ge j_0.
  \label{Ry2-ineq}
  \end{equation}
\end{lem}
\begin{proof}
  With $D_{j,N}=\sup_{k\ge j}2^{(j-k)N}d_k$ it follows from \eqref{Ry1-ineq} that for 
  $j\ge j_0$, $N>0$,
  \begin{equation}
    D_{j,N}\le \sup_{k\ge j} C_N\sum_{l\ge k} 2^{(j-l)N}b_ld_l^{1-r}
          \le  C_N(\sum_{l\ge j} 2^{(j-l)Nr}b_l)D_{j,N}^{1-r}.
\label{Ry3-ineq}
  \end{equation}
Clearly $D_{j_1,N}=0$ implies $d_j=0$ for $j\ge j_1$, so \eqref{Ry2-ineq} is trivial for
such $j$. We thus only need to consider the $D_{j,N}>0$.
Now \eqref{Ry0-ineq} yields that $D_{j,N}<\infty$ for all $j\ge j_0$ when $N\ge N_0$, so
then \eqref{Ry2-ineq} follows from \eqref{Ry3-ineq} by division by $D_{j,N}^{1-r}$.
 
Given any $N\in\,]0,N_0[\,$, we may in the just proved cases of \eqref{Ry2-ineq} decrease $N_0$ to
$N$, which gives a version of \eqref{Ry2-ineq} with $N$ in the exponent and the constant
$C_{N_0}$. Analogously to \eqref{Ry3-ineq}, one therefore finds from the definition of $D_{j,N}$ that
$D_{j,N}\le C_{N_0}^{1/r}(\sum_{l\ge j}2^{(j-l)Nr}b_l)^{1/r}$ for $j\ge j_0$. Here the right-hand
side may be assumed finite (as else \eqref{Ry2-ineq} is trivial for this $N$), whence we may
proceed as before by division in \eqref{Ry3-ineq}.
\end{proof}

\begin{rem}
  \label{Rych1-rem}
Lemma~\ref{sum2-lem} was essentially crystallised by Rychkov~{\cite[Lem.~3]{Ry99BPT}}, 
albeit with three unnecessary assumptions: 
$d_j<\infty$ (a consequence of \eqref{Ry0-ineq}), 
that $b_j,d_j>0$ and that $j_0=0$. 
For our proof of Proposition~\ref{prop:pointwiseEstimate} below, it is essential to consider 
$j_0>0$, and it would be cumbersome there to reduce to strict positivity of $b_j, d_j$.  
In \cite{Ry99BPT} no justification was given for this strictness in the application 
of \cite[Lem.~3]{Ry99BPT}, but this is remedied by Lemma~\ref{sum2-lem} above.
\end{rem}
\section{Some maximal inequalities} \label{max-sect}
In this section we obtain some maximal inequalities in the mixed-norm set-up. 
This part of the theory of the $F^{s,\vec a}_{\vec p,q}$-spaces is interesting in its own
right, and also important for the authors' work \cite{JoMHSi2}. Moreover, the methods are 
similar to those adopted in the set-up in Section~\ref{Rych-sect} below, but are rather cleaner
here. 

For distributions $u$ that for some $R>0$ and $j\in{\mathbb{N}}$ satisfy
\begin{equation}
 \supp \widehat{u}\subset \big\{ \xi\in\Rn \bigm| |\xi_k|\leq R\, 2^{j a_k}, k=1,\ldots,n \big\}
\label{spu-eq}
\end{equation}
the Peetre-Fefferman-Stein maximal function $u^* (x)$ is given by
\begin{equation}
  u^* (x)= \sup_{y \in \Rn}\, \frac{|u(y)|}{\prod\limits_{l=1}^n
  (1+ R\, 2^{ja_l} |x_l - y_l|)^{r_l}},\quad \vec{r}>0.
  \label{PFS-id}
\end{equation}
It obviously fulfils
\begin{equation}\label{eq:finitenessMaxFct}
|u(x)| \leq u^*(x) \leq \|\, u\, | L_\infty \|, \quad x\in\Rn.
\end{equation}
When $u$ in addition is in $L_{\vec{p}}$, the Nikol'ski\u\i--Plancherel--Polya inequality for
mixed norms, cf. \cite[Prop.~4]{JoSi07},  gives the finiteness of the right-hand side, hence
$u^*$ is finite everywhere. Thus, analogously to \cite[Sec.~2]{JJ11pe}, 
the maximal function is continuous.

To prepare for the theorem below, we first show the following pointwise estimate of
$u^*(x)$ by combining 
the proof ingredients from \cite[Prop.~2.2]{JJ11pe}, which the reader may consult for 
more details. Now their order is crucial:

\begin{prop}\label{prop:pointwiseEstMax}
When $0<\vec{q},\vec{r}\leq \infty$ then there is a constant $c_{\vec q,\vec r}$ such that
every $u\in\cs'$ fulfilling \eqref{spu-eq} also satisfies
\begin{equation}
  u^*(x) \leq c_{\vec{q},\vec{r}} 
  \left\| \frac{ u(x-R^{-1}\, 2^{-j\vec{a}}z) }{ \prod_{l=1}^n (1+|z_l|)^{r_l}}\, 
  \middle| L_{\vec{q}}\,(\R^n_z)\right\|
  \qquad\text{for $x\in\Rn$}.
\label{eq:generalProp22}
\end{equation}
\end{prop}

\begin{proof}
Taking $\psi\in\cs(\Rn)$ with $\widehat{\psi} \equiv 1$ on $[-1,1]^n$ 
and such that $\supp \widehat{\psi} \subset [-2,2]^n$, 
we have $u=\cfi(\hat\psi(R^{-1}2^{-j\vec a}\cdot))*u$,  
which may be written with an integral since $u$ is $C^\infty$ with polynomial growth, i.e.\ 
\begin{equation}\label{eq:uAsAnIntegral}
  u(y)=\int\cdots\int R^n\, 2^{j|\vec{a}|}\, \psi(R\, 2^{j\vec{a}}(y-z))\, u(z)\, dz_1 \cdots dz_n .
\end{equation}
Now $\vec{q}=(q_<,q_\geq)$ is split into two groups $q_<$ and $q_\geq$ according to whether
$q_k<1$ or $q_k\geq 1$ holds. The groups may be interlaced, but for simplicity this is
ignored in the notation; the important thing is to treat the two groups separately. 

First \eqref{eq:uAsAnIntegral} is estimated by the norm of $L_1(\Rn)$, which then is controlled in
terms of the norm of $L_{(q_<,1_\geq)}$, whereby interlacing of the groups $q_<$ and $1_\geq$ is
unimportant: 
for fixed $y$, the spectrum of the integrand in \eqref{eq:uAsAnIntegral} is contained in 
$[-3R\, 2^{ja_1},3R\, 2^{ja_1}]\times\ldots\times[-3R\, 2^{ja_n},3R\, 2^{ja_n}]$, so 
the Nikol'ski\u\i--Plancherel--Polya inequality for mixed norms applies, 
cf. \cite[Prop.~4]{JoSi07}, which for $q_k<1$ gives an estimate by the $L_{q_k}$-norm with
respect to $z_k$,
\begin{equation}
  |u(y)| \leq c \prod_{q_k<1} (3R\, 2^{ja_k})^{\frac{1}{q_k}-1}\, 
  \big\| R^n2^{j|\vec{a}|} \psi(R2^{j\vec{a}}(y-\cdot))\, u\, \big| L_{(q_<,1_\geq)} \big\|.
\end{equation}
(The integration order in this norm is as stated in \eqref{eq:uAsAnIntegral}.)

Secondly, using H\"{o}lder's inequality in the variables where $q_k\geq 1$, and gathering their
dual exponents $q_k^*$ in $(q_\geq)^*$, gives for $x\in\Rn$ 
\begin{multline}
  \frac{|u(y)|}{\prod_l (1+R2^{ja_l}|x_l-y_l|)^{r_l}}
  \leq c \prod_{q_k<1} (3R 2^{ja_k})^{\frac{1}{q_k}-1}
  \Big\| \frac{R^n2^{j|\vec{a}|} u(z)}{\prod_l (1+R 2^{ja_l}|x_l-z_l|)^{r_l}} \, \Big| L_{\vec{q}} 
  \Big\|
\\
  \phantom{\leq}\;\times\Big\| \prod_l (1+R2^{ja_l}|y_l-z_l|)^{r_l} \psi(R 2^{j\vec{a}}(y-z))\, 
  \Big| L_{(\infty_<,(q_\geq)^*)} \Big\|.
\end{multline}
Since $\psi\in\cs$, a change of coordinates $z_k\mapsto R^{-1}\, 2^{-ja_k}z_k$ yields \eqref{eq:generalProp22} with the constant
$c_{\vec{q},\vec{r}} = c \prod_{q_k<1} 3^{\frac{1}{q_k}-1} \| \prod_{l=1}^n (1+|z_l|)^{r_l} \psi \, | L_{(\infty_<,(q_\geq)^*)} \|<\infty$.
\end{proof}

We now obtain an elementary proof of the \emph{mixed-norm} boundedness of $u^*$, by adapting the
proof of the isotropic $L_p$-result in \cite[Thm.~2.1]{JJ11pe}:

\begin{thm}\label{maxx}
Let $0<\vec{p}\leq \infty$ and suppose
\begin{equation}\label{eq-1}
  r_l > \frac{1}{\min (p_1, \ldots\, , p_l)}\, , \quad l = 1, \ldots, n .
\end{equation}
Then there exists a constant $c$ such that
\begin{equation}
\label{eq:maxxInequality}
  \| \, u^* \, | L_{\vec{p}}\| \le c \, \| \, u\, | L_{\vec{p}}\|
\end{equation}
holds for all $u \in L_{\vec{p}} \cap \cs' $ satisfying the spectral condition \eqref{spu-eq}.
\end{thm}

\begin{proof}
We use \eqref{eq:generalProp22} with $q_k=\min(p_1,\ldots,p_k)$ for $k=1,\ldots,n$ and calculate the $L_{p_j}$-norms successively on both sides. 
Since $p_j\geq q_k$ for all $k\geq j$, we may apply the generalised Minkowski inequality
$n-(j-1)$ times, as well as the  
translation invariance of $dx_1,\ldots,dx_n$, which gives
\begin{equation}
  \|\, u^*\, | L_{\vec{p}} \| \leq c_{\vec{q},\vec{r}} 
  \big(\prod_{l=1}^n \| (1+|z_l|)^{-r_l}\, | L_{q_l}\|\big )
  \|\, u\, | L_{\vec{p}}\|.
\end{equation}
Here \eqref{eq-1} yields the finiteness of the $L_{q_l}$-norms.
\end{proof}

The following result is convenient for certain convolution estimates.
Since the embedding $B^{s,\vec a}_{\vec p,q}(\Rn)\hookrightarrow C^0(\Rn)\cap L_\infty(\Rn)$
holds for $s>\vec a\cdot\frac1{\vec p}$, 
or for $s=\vec a\cdot\frac1{\vec p}$ if $q\le1$, it is a result 
pertaining to continuous functions.

\begin{cor}\label{osz}
If $C>0$ and $\vec{r}$ fulfils \eqref{eq-1}, 
$d=\min (1,p_1,\ldots ,p_n)$ yields
\begin{equation}
\label{eq:supBall}
\Big\|\, \sup_{|x-y| < C} |u(y)| \, \Big| L_{\vec{p}}(\R^n_x)\Big\|\le c \, 
\|\, u\, | B^{s,\vec{a}}_{\vec{p},d}\|
\qquad\text{for $s=\vec{a}\cdot\vec{r}$}.
\end{equation}
\end{cor}

\begin{proof}
Since $\|\cdot |L_{\vec p}\|^d$ is subadditive,
simple arguments yield
\begin{equation}
\begin{split}
  \Big\|\, \sup_{|x-y| < C} |u(y)| \, \Big| L_{\vec{p}}(\Rn_x) \Big\|^d & \le 
  \Big\|\, \sup_{|x-y| < C} \sum_{j=0}^\infty \, |u_j(y)| \, \Big|\, L_{\vec{p}}(\Rn_x)\Big\|^d 
  \\
  & \le  \sum_{j=0}^\infty \, 
  \prod_{\ell =1}^n (1+C\, 2^{ja_\ell})^{d\, r_\ell}
  \| \, u_j^*
   \, | L_{\vec{p}}\|^d.
\end{split}
\end{equation}
Since $\prod_{\ell=1}^n (1+C\, 2^{ja_\ell})^{d\, r_\ell} \leq (1+C)^{d\, |\vec{r}|}\, 2^{jd\, \vec{a}\cdot\vec{r}}$,
the right-hand side is seen to be less than  
$c\, \|\, u \, |\, B^{s,\vec{a}}_{\vec{p},d}\|^d$ for $s=\vec a\cdot\vec r$
by application of Theorem~\ref{maxx}.
\end{proof}

\begin{rem}\label{rem:lpLTEstimate}
In \cite{JoMHSi2}  Corollary~\ref{osz} enters our estimates 
for certain functions $u\in F^{s,\vec a}_{\vec p,q}$ with
$s>\sum_{\ell=1}^n \frac{a_\ell}{\min (p_1, \ldots\, , p_\ell)}$.
Then one can pick $\vec{r}$ satisfying \eqref{eq-1} and such that $\vec{a}\cdot\vec{r}<s$,
hence elementary embeddings yield
$\|\, \sup_{|x-y| < C} |u(y)| \, | L_{\vec{p}}(\R^n_x)\|\le c \, 
\|\, u \, |  F^{s,\vec{a}}_{\vec{p},q}\|$.
\end{rem}

\section{Rychkov's inequalities}
\label{Rych-sect}

In the systematic theory of the $F^{s,\vec a}_{\vec p,q}$-spaces, it is of course important
to dispense from the requirement in Definition~\ref{F-defn} that the Schwartz functions $\Phi_j$
have compact support. In so doing, we shall largely follow Rychkov's treatment of the isotropic case
\cite{Ry99BPT}. 

In the following $\vec{a}=(a_1, \ldots, a_n)$ is a fixed anisotropy with $\vec{a}\geq 1$; we set
\begin{equation}
  a_0 = \min (a_1, \ldots ,a_n)  .
\end{equation}
Throughout this section we consider $\psi_0, \psi \in \cs (\Rn)$ that 
fulfil Tauberian conditions in terms of some $\varepsilon >0$ and/or a moment 
condition of order $M_\psi$,
\begin{alignat}{2}
\label{con1}
|{\mathcal{F}} \psi_0 (\xi) | &> 0& \quad &\mbox{on }\{\xi \, |\,  |\xi|_{\vec{a}}< 2 \varepsilon\},
\\
\label{con2}
|{\mathcal{F}} \psi (\xi) | &> 0& \quad &\mbox{on }\Big\{\xi \, \Big|\,  
\frac{\varepsilon}{2}<|\xi|_{\vec{a}}< 2 \varepsilon\Big\} ,\\
\label{moment}
D^\alpha ({\mathcal{F}} \psi) (0)&= 0& \quad &\mbox{for }|\alpha|\le M_\psi.
\end{alignat}
Hereby $M_\psi\in{\mathbb{N}}_0$, or we take $M_\psi=-1$ when the condition \eqref{moment} is void.
Note that if \eqref{con1} is verified for the Euclidean distance, it holds true also 
in the anisotropic case, perhaps with a different $\varepsilon$; cf.~\eqref{ax-ineq}.

In this section we also change notation by setting 
\begin{equation}\label{eq:defOfSubscript_j}
  \varphi_j(x) = 2^{j|\vec a|}\varphi(2^{j\vec a}x),\quad \varphi\in\cs,\enskip j\in\N.
\end{equation} 
For $\psi_0$ this gives rise to the sequence $\psi_{0,j}(x):=2^{j|\vec a|} \psi_0(2^{j\vec a}x)$,
but we shall mainly deal with $(\psi_j)_{j\in\N_0}$ that mixes
$\psi_0$ and $\psi$. Note that $\psi_0=\psi_{0,0}$.  

To elucidate the Tauberian conditions, we recall in the lemma below a well-known fact on Calder\'{o}n's
reproducing formula:
\begin{equation}\label{eq:calderonRP}
  u = \sum_{j=0}^\infty \lambda_j * \psi_j * u,
  \qquad\text{for $u\in\cs'(\Rn)$}.
\end{equation}

\begin{lem}\label{lem:calderonRP}
When $\psi_0,\psi\in\cs(\Rn)$ fulfil the Tauberian conditions
\eqref{con1}--\eqref{con2} there exist $\lambda_0,\lambda\in\cs(\Rn)$ 
fulfilling \eqref{eq:calderonRP} for every $u\in\cs'(\Rn)$.
Moreover, it can be arranged that $\widehat\lambda_0$ and
$\widehat\lambda$ are supported by the sets in 
\eqref{con1}, respectively \eqref{con2}. 
\end{lem}

\begin{proof}
By Fourier transformation \eqref{eq:calderonRP} is carried over to
\begin{equation}\label{eq:identityLambda}
\cf \lambda_0 (\xi) \, \cf \psi_0 (\xi) + \sum_{j=1}^\infty
\cf \lambda (2^{-j\vec{a}}\xi) \, \cf \psi (2^{-j\vec{a}}\xi) =1 ,
\quad \xi \in \Rn .
\end{equation}
To find $\lambda_0,\lambda$ reduces to a Littlewood-Paley construction:
taking $h\in C_0^\infty$
such that $0\leq h\leq 1$ on $\Rn$, 
$\supp h\subset\{\xi\, |\, |\xi|_{\vec{a}}< 2\varepsilon\}$ and 
$h(\xi)=1$ if $|\xi|_{\vec{a}}\leq \frac{3}{2}\varepsilon$, then 
$\widehat{\lambda_0}=h/\widehat{\psi_0}$
and 
$\widehat{\lambda} = \big(
h-h(2^{\vec{a}}\cdot)\big)/\widehat{\psi}$ 
fulfil \eqref{eq:identityLambda} and the support inclusions.
\end{proof}
A general reference to Calderon's formula could be \cite[Ch.~6]{FJW91}.
More refined versions have been introduced by Rychkov~\cite{Ryc01}.

To comment on the moment condition, we use
for $M \ge -1$ the subspace
\begin{equation}\label{eq:momentConditionSM}
\cs_M := \Big\{\mu \in \cs(\Rn) \, \Big| \, D^\alpha (\cf \mu)(0)=0 \quad \mbox{for all}\quad
|\alpha| \le M \Big\} .
\end{equation}
It is recalled that in addition to the $p_{\alpha,\beta}$ in
\eqref{eq:seminormsSchwartz} also the following family of  
seminorms induces the topology on $\cs$:
\begin{equation}\label{eq:seminormsS_Nalpha}
q_{N,\alpha}(\psi) := \int_{\R^n} \langle x\rangle^N |D^\alpha\psi(x)|\, dx,
\quad N\in\N_0,\enskip \alpha\in\N_0^n.
\end{equation}
This is convenient for the fact that moment conditions,
also in case of the anisotropic dilation $t^{\vec a}$,
induce a rate of convergence to 0 in $\cs$:

\begin{lem}\label{Bui}
For $\alpha,\beta\in\N_0^n$ there is an estimate for $0<t\leq 1, \nu\in\cs$ and $\mu\in\cs_M$, 
\begin{equation}
  p_{\alpha,\beta}(t^{-|\vec{a}|}\mu(t^{-\vec{a}}\, \cdot) *\nu) \leq C_{\alpha}\, t^{(M+1)a_0} \,
  \max p_{0,\zeta}(\widehat{\mu})\cdot q_{M+1,\gamma}(\widehat{D^\beta \nu}),
\end{equation}
where the maximum is over all $\zeta$ with $|\zeta|\leq M+1$ or $\zeta\leq \alpha$; and over
$\gamma\leq\alpha$. 
\end{lem}

\begin{proof}
The continuity of $\cfi=(2\pi)^{-n} \overline{\cf}: L_1\to L_\infty$ and Leibniz' rule give that
\begin{gather}
\begin{split}
\label{eq:seminormsConvolution}
    p_{\alpha,\beta}(t^{-|\vec{a}|}\mu(t^{-\vec{a}}\, \cdot) * \nu)
   &= \sup_{z\in\R^n} \Big|\cfi \left( D_\xi^\alpha (t^{-|\vec{a}|}\widehat{\mu(t^{-\vec{a}}\, \cdot)}\, \widehat{D^\beta\nu} )\right)(z) \Big|\\
  &\leq \sum_{\gamma\leq\alpha} \binom{\alpha}{\gamma} \int t^{a\cdot(\alpha-\gamma)} |D^{\alpha-\gamma} 
  \widehat{\mu}(t^{\vec{a}}\xi)|\, |D^\gamma \widehat{D^\beta\nu}(\xi)|\, d\xi.
\end{split}
\end{gather}
For $|\alpha-\gamma|\leq M$ the integral is estimated using a Taylor expansion of order $N:=M-|\alpha-\gamma|$. 
All terms except the remainder vanish, because $\mu$ has vanishing moments up to order $M$. 
The integral is therefore bounded by
\begin{alignat}{1}
   \int t^{\vec{a}\cdot (\alpha-\gamma)} &\Big| \sum_{|\zeta|=N+1} 
   \frac{N+1}{\zeta!} (t^{\vec{a}}\xi)^\zeta \int_0^1 (1-\theta)^N 
  \partial_\xi^\zeta D_\xi^{\alpha-\gamma} \widehat{\mu}(\theta t^{\vec{a}}\xi) d\theta\, \Big|\,
  |D^\gamma \widehat{D^\beta\nu}(\xi)|\, d\xi
\notag \\
  &\leq  t^{(M+1)a_0}\, \max_{|\zeta|\leq M+1} \| D^\zeta \widehat{\mu}\, | \, L_\infty\| 
  \int |\xi|^{N+1} |D^\gamma \widehat{D^\beta\nu}(\xi)|\, d\xi
\notag \\
  &\leq  t^{(M+1)a_0}\, \max_{|\zeta|\leq M+1} p_{0,\zeta}(\widehat{\mu}) \, q_{M+1,\gamma} (\widehat{D^\beta\nu}).
\end{alignat}
For $|\alpha-\gamma|\geq M+1$ the integral in
\eqref{eq:seminormsConvolution} is easily seen to be estimated by 
\begin{equation}
  t^{(M+1)a_0}\, \max_{\zeta \leq \alpha} p_{0,\zeta} (\widehat{\mu})\, q_{0,\gamma}(\widehat{D^\beta\nu}).
\end{equation}
The claim is obtained by taking the largest of the bounds.
\end{proof}

\subsection{Comparison of norms}
For any $\vec{r} = (r_1, \ldots \, , r_n)>0$ and $f \in \cs'(\Rn)$ 
we deal in this section with the non-linear maximal operators of Peetre-Fefferman-Stein type
induced by $\{\psi_j\}_{j\in\N_0}$, 
\begin{alignat}{1}
\label{eq:maximalPFStype}
 \psi^*_j f(x)  =  \sup_{y \in \Rn}\, 
 \frac{|\psi_j * f (y)|}{\prod\limits_{\ell=1}^n
(1+ 2^{ja_\ell}\,|x_\ell - y_\ell |)^{r_\ell}} \, , \quad x\in \Rn,\enskip j \in\N_0  .
\end{alignat}
For simplicity their dependence on $\vec{a}$ and $\vec{r}$ is omitted.
(Compared to \eqref{PFS-id}, no $R$ is in the denominator here,
as $\psi_j*f$ need not have compact spectrum.)

To give the background, we recall an important technical result of Rychkov:

\begin{prop}[{\cite[(8')]{Ry99BPT}}]\label{max}
Let $\psi_0,\psi\in\cs(\Rn)$ be given such that \eqref{moment} holds, while
$\varphi_0,\varphi\in\cs(\Rn)$ 
fulfil the Tauberian conditions \eqref{con1}--\eqref{con2} in terms of some $\varepsilon'>0$.
When $0< p < \infty$, $0<q\le \infty$ and $s < (M_\psi+1)\, a_0$
there exists a constant $c>0$ such that for $f \in \cs' (\Rn)$,
\begin{equation}
 \| 2^{sj} \psi^*_j f \, | L_p(\ell_q)\|
\le c\,                     
 \| 2^{sj} \varphi^*_j f \, | L_p(\ell_q)\|.
\end{equation}
\end{prop}

We shall extend this to a mixed-norm version, which even covers
parameter-dependent families of the spectral cut-off functions; this will be
crucial for our results in \cite{JoMHSi2}.
So if $\Theta$ denotes an index set and $\psi_{\theta,0},\psi_\theta\in\S(\Rn)$,
$\theta\in\Theta$, we set $\psi_{\theta,j}(x):=2^{j|\vec a|}\psi_\theta(2^{j\vec a}x)$ for 
$j\in\N$.
Not surprisingly we need to assume that the $\psi_\theta$ fulfil the same moment condition, i.e.\ 
uniformly with respect to $\theta$:

\begin{thm}\label{maxmod}
Let $\psi_{\theta,0},\psi_\theta\in\cs(\Rn)$ be given such that \eqref{moment} 
holds for some $M_{\psi_\theta}$ independent of $\theta\in \Theta$, while $\varphi_0,\varphi\in\cs(\Rn)$
fulfil \eqref{con1}--\eqref{con2} in terms of an $\varepsilon'>0$.
Also let $0< \vec{p}< \infty$, $0<q\le \infty$ and $s < (M_{\psi_\theta}+1)\, a_0$. 
For a given $\vec{r}$ in \eqref{eq:maximalPFStype}
and an integer $M\geq -1$ chosen so large that $(M+1)a_0+s>2\vec{a}\cdot\vec{r}$, we assume that
\begin{alignat*}{2}
A  & :=  \sup_{\theta \in \Theta}\,  \max  \| \, D^\alpha \cf \psi_\theta\, | L_\infty\|&\, <\infty ,
\\
B & : =  \sup_{\theta \in \Theta}\,  \max \| \, (1+ |\xi|)^{M+1}\, D^\gamma \cf \psi_\theta (\xi)\, | L_1\|&\, <\infty ,
\\
C & :=  \sup_{\theta \in \Theta}\,  \max \| \, D^\alpha \cf \psi_{\theta,0}\, | L_\infty\|&\, <\infty ,
\\
D  & : =  \sup_{\theta \in \Theta}\,  \max \| \, (1+ |\xi|)^{M+1}\, D^\gamma \cf \psi_{\theta,0} (\xi)\, | L_1\|&\, <\infty,
\end{alignat*}
where the maxima are over all $\alpha$ with $|\alpha|\leq M_{\psi_\theta}+1$ or $\alpha\leq
\lceil\vec r + 2\rceil$,  
respectively over $\gamma\leq \lceil\vec r+2\rceil$.
Then there exists a constant $c>0$ such that for $f \in \cs' (\Rn)$,
\begin{gather}
\begin{split}
\label{eq-8}
  \| 2^{sj} \sup_{\theta \in \Theta}\, \psi_{\theta,j}^* f \, | L_{\vec{p}}(\ell_q)\|
  \le  c (A +B + C +D)\, \| 2^{sj} \varphi^*_j f \, | L_{\vec{p}}(\ell_q)\|.
\end{split}
\end{gather}
\end{thm}
Hereby $\lceil t\rceil$ denotes the smallest integer $k\geq t$, 
and  $\lceil\vec r \rceil:= \left(\lceil r_1\rceil,\ldots,\lceil r_n\rceil\right)$.

In the proof of the estimate \eqref{eq-8} we choose $\lambda_0, \lambda \in \cs (\Rn)$ 
by applying Lemma~\ref{lem:calderonRP} to the given $\varphi_0,\varphi\in\cs(\Rn)$.
Following \cite{Ry99BPT}, we then consider the auxiliary integrals
\begin{equation}
  I_{j,k}:= \int \, |\psi_{\theta,j} * \lambda_k (z)|\, \prod\limits_{\ell=1}^n
  (1+ 2^{k a_\ell}\,|z_\ell |)^{r_\ell}\, dz\, ,\quad j,k \in \N_0.
\label{ijk-id}
\end{equation}
The integrand may be estimated using that $\psi_{\theta,j} * \lambda_k (z)=
2^{k|\vec a|}\psi_{\theta,j-k} * \lambda (2^{k\vec a}z)$, so
the Binomial Theorem and Lemma~\ref{Bui} with $\beta=0$, $t^{-1}=2^{j-k}\ge1$ yield
\begin{equation}
\begin{split}\label{binomEstimate}
  |\psi_{\theta,j} * \lambda_k(z)| \prod_{l=1}^n (1+&|2^{ka_l}z_l|)^{r_l}
  \leq 2^{k|\vec a|} \sum_{\alpha\leq\lceil\vec r\rceil} \binom{\lceil\vec r\rceil}{\alpha}
  p_{\alpha,0} (\psi_{\theta,j-k}*\lambda) \\ 
  &\leq C_{\lceil\vec r\rceil} 2^{(k-j)(M_{\psi_\theta}+1)a_0+k|\vec a|} 
  \sideset{}{'}\max p_{0,\zeta}(\widehat{\psi_\theta}) \cdot q_{M_{\psi_\theta}+1,\gamma}(\widehat{\lambda}),
\end{split}
\end{equation}
where $\sideset{}{'}\max$ denotes a maximum over finitely many multi-indices, in this case
over $\zeta$ fulfilling $|\zeta|\leq M_{\psi_\theta}+1$ or $\zeta\leq\lceil\vec r\rceil$,
respectively $\gamma \leq \vec r$. 

\begin{lem}\label{ijk}
For any integer $M\geq -1$ there exists a constant $c=c_{M,M_\psi,\vec{r},\lambda_0,\lambda}$ such that
for $k,j \in {\mathbb{N}}_0$,
\begin{equation}
  I_{j,k}\le c \, (A+B+C+D)\times 
  \left\{
  \begin{array}{lll}
  2^{(k-j)(M_{\psi_\theta}+1)a_0} & & \mbox{for}\quad k \le j , \\
  2^{-(k-j)((M+1)a_0-\vec{a}\cdot\vec{r}\, )} & & \mbox{for}\quad j \le k\,  
  \end{array}
  \right.
\end{equation}
when $\psi_{\theta,0},\psi_\theta\in\S$ and the
$\psi_\theta$ fulfil \eqref{moment} for some $M_{\psi_\theta}$ independent of $\theta\in \Theta$.
\end{lem}

\begin{proof} 
First we consider the case $j\geq k\geq 1$, where \eqref{binomEstimate} yields
\begin{equation}
\begin{split}
  I_{j,k} &\leq \sup_{z\in\Rn} |\psi_{\theta,j} * \lambda_k(z)| \prod_{l=1}^n
  (1+2^{ka_l}|z_l|)^{r_l+2} \int \prod_{l=1}^n 2^{-ka_l}(1+|x_l|)^{-2}\, dx
\\ 
  &\leq C_{\vec{r}}\, 2^{(k-j)(M_{\psi_\theta}+1)a_0} \sideset{}{'}\max \|\, D^\zeta
  \widehat{\psi_\theta}\, |\, L_\infty\| \cdot q_{M_{\psi_\theta}+1,\gamma}(\widehat{\lambda})
\\ 
  &\leq C_{\vec{r},M_{\psi_\theta},\lambda}\, 2^{(k-j)(M_{\psi_\theta}+1)a_0} A.
\end{split} 
\end{equation}
For $k\geq j\geq 1$ one can replace $2^{ka_l}$ in \eqref{ijk-id} by $2^{ja_l}$ at the cost of
the factor $2^{(k-j)\vec a\cdot\vec r}$ in front of the integral.
Then the roles of $\psi_\theta$ and $\lambda$ can be interchanged, 
since the support information on $\widehat\lambda$ yields $\lambda\in \bigcap_M \cs_M$.
This gives, with $\rho=\lceil\vec{r}+2\rceil$,
\begin{equation}
  I_{j,k} \leq c2^{(k-j)\vec{a}\cdot\vec{r}} \sum_{\alpha \leq \rho} 
  \binom{\rho}{\alpha}\, p_{\alpha,0}(\psi_\theta*\lambda_{k-j})
  \leq C_{M,\vec{r},\lambda} 2^{-(k-j)((M+1)a_0-\vec{a}\cdot\vec{r}\, )} B.
\end{equation}
Similar estimates are obtained for $I_{j,0}$, $I_{0,k}$ and $I_{0,0}$ with $C$, $D$ as factors.
\end{proof}

Using Lemma~\ref{ijk}, the proof given in \cite{Ry99BPT} is now extended to a

\begin{proof}[Proof of Theorem~\ref{maxmod}]
The identity \eqref{eq:calderonRP} gives for $f\in\cs'$ and $j\in\N$ that
\begin{equation}\label{eq:psiThetaf}
\psi_{\theta,j}*f = \sum_{k=0}^\infty \psi_{\theta,j}*\lambda_k*\varphi_k*f.
\end{equation}
By Lemma~\ref{ijk} with $M$ chosen so large that $(M+1)a_0+s>2\vec{a}\cdot\vec{r}$, 
there exists a $\theta$-independent constant $c>0$ such that
the summands can be crudely estimated,
\begin{equation}
  \begin{split}
  |\psi_{\theta,j}*\lambda_k*{}&\varphi_k*f(y)| 
  \leq \varphi_k^* f(y) \int |\psi_{\theta,j} *\lambda_k(z)| \prod_{l=1}^n (1+2^{ka_l}|z_l|)^{r_l}\, dz\\
  &\leq  c\, (A +B + C +D)\, \varphi_k^* f(y) \times
  \left\{\begin{array}{lll}
    2^{(k-j)(M_{\psi_\theta}+1)a_0} &  \mbox{for}\enskip k \le j, \\
    2^{-(k-j)((M+1)a_0-\vec{a}\cdot\vec{r}\, )} &  \mbox{for}\enskip j \le k . 
  \end{array}\right.
  \end{split}
\end{equation}
Here
$  \varphi_k^*f(y) \leq \varphi_k^*f(x) 
\max    \big(1,2^{(k-j)\vec{a}\cdot\vec{r}}\big)\prod_{l=1}^n(1+2^{ja_l}|x_l-y_l|)^{r_l}
$
is easily verified for $x,y\in{\mathbb{R}}^n$ and $j,k\in{\mathbb{N}}_0$ by elementary calculations,
so 
therefore
\begin{equation}
\begin{split}
  \sup_{y\in\Rn}{}&\frac{|\psi_{\theta,j} *\lambda_k * \varphi_k*f(y)|}{\prod_{l=1}^n
  (1+2^{ja_l}|x_l-y_l|)^{r_l}}
\\ 
  &\qquad\ \leq c (A +B + C +D)\, \varphi_k^* f(x) \times 
  \left\{\begin{array}{lll}
  2^{(k-j)(M_{\psi_\theta}+1)a_0} &  \mbox{for}\enskip k \le j, \\ 
  2^{-(k-j)((M+1)a_0-2\vec{a}\cdot\vec{r}\, )} &  \mbox{for}\enskip j \le k . 
  \end{array}\right.  
\end{split}
\end{equation}
Inserting into \eqref{eq:psiThetaf} and using that 
$\delta:=\min((M_{\psi_\theta}+1)a_0-s,(M+1)a_0-2\vec{a}\cdot\vec{r}+s)>0$ by the assumptions, the
above implies for $j\geq 0$, 
\begin{equation}
  2^{js} \sup_{\theta\in \Theta} \psi_{\theta,j}^* f(x)
  \leq c (A +B + C +D) \sum_{k=0}^\infty 2^{ks}\varphi_k^*f(x)\, 2^{-|j-k|\delta}.
\end{equation}
Now Lemma~\ref{sum1-lem} yields \eqref{eq-8}.
\end{proof}

\subsection{Control by convolutions}
Since $\widehat{\psi}$ need not have compact support,
Proposition~\ref{prop:pointwiseEstMax} is replaced by a pointwise
estimate with a sum representing the higher frequencies:

\begin{prop}\label{prop:pointwiseEstimate}
Let $\psi_0,\psi\in\cs(\Rn)$ satisfy the Tauberian conditions \eqref{con1}--\eqref{con2}. 
For $N,\vec r, \tau>0$ there exists a constant $C_{N,\vec{r},\tau}$ 
such that for $f\in\cs'(\Rn)$ and $j\in\N_0$,
\begin{alignat}{1}
\label{eq:maxFuncEstimate}
\left(\psi_j^*f(x)\right)^\tau \leq C_{N,\vec{r},\tau}\, \sum_{k\geq j} 2^{(j-k)N\tau} 
\int \frac{2^{k|\vec{a}|} |\psi_k*f(z)|^\tau}{\prod_{l=1}^n (1+2^{ka_l}|x_l-z_l|)^{r_l \tau}} \, dz.
\end{alignat}
\end{prop}

As a proof ingredient we use the $\cs'$-order of $f\in\cs'(\Rn)$, written
$\operatorname{ord}_{\cs'}(f)$,  
that is the smallest $N\in\N_0$ for which there exists
$c>0$ such that, cf. \eqref{eq:seminormOneParam}, 
\begin{equation}
  |\langle f,\psi\rangle| \leq c\, p_N(\psi)\quad\text{ for all $\psi\in\cs(\Rn)$}.  
\label{ordS'-id}
\end{equation}

\begin{rem} \label{Rych-rem}
Our proof of Proposition~\ref{prop:pointwiseEstimate} follows that of
Rychkov~\cite{Ry99BPT}, although his exposition leaves a heavy burden with the
reader, since the application of Lemma~3 
there is only justified when $\operatorname{ord}_{\cs'}(f)$ is sufficiently small;
cf.\ the $O$-condition \eqref{ON0-eq} in Step~2 below. In a somewhat different
context, Rychkov gave a verbal explanation after (2.17) in \cite{Ryc01}
(with similar reasoning in \cite{ST89,Han10}) that perhaps could be carried over 
to the present situation. But we have found it simplest 
to reinforce \cite{Ry99BPT} by showing that the central $O$-condition 
\emph{is} indeed fulfilled whenever $f$ is such that the right-hand side of 
\eqref{eq:maxFuncEstimate} is finite.
In so doing, we give the full argument for the sake of completeness.
\end{rem}

\begin{proof}
\emph{Step~1.} First we choose two functions $\lambda_0,\lambda\in\cs$ with $\hat\lambda=0$ around $\xi=0$ by applying Lemma~\ref{lem:calderonRP} to the given $\psi_0,\psi\in\cs(\Rn)$.
Using Calder\'{o}n's reproducing formula, cf. \eqref{eq:calderonRP}, 
on $f(2^{-j\vec{a}}\cdot)$, dilating and convolving with $\psi_j$, we obtain
\begin{equation}\label{eq:prop3psiConvf}
  \psi_j*f = (\lambda_{0,j}*\psi_{0,j})*(\psi_j*f)+\sum_{k=j+1}^\infty (\psi_j*\lambda_k)*(\psi_k*f).
\end{equation}
To estimate $\psi_j*\lambda_k$ we use  \eqref{binomEstimate} for an arbitrary 
integer $M_\lambda \geq -1$ to get
\begin{alignat}{1}
  |\psi_j*\lambda_k(z)|  
  \leq C_{\vec{r}}\, \frac{2^{j|\vec{a}|}\, 2^{(j-k)(M_\lambda+1)a_0}}{\prod_{l=1}^{n}(1+2^{ja_l}|z_l|)^{r_l}}  
  \sideset{}{'}\max
  p_{0,\zeta}(\widehat{\lambda}) \cdot q_{M_\lambda+1,\gamma}(\widehat{\psi}).
\end{alignat}
An analogous estimate is obtained for $\lambda_{0,j}*\psi_{0,j}$, 
when \eqref{binomEstimate} is applied with  $t=1$, $M_{\lambda_0}=-1$.
Inserting these bounds into \eqref{eq:prop3psiConvf} yields for
$C_{M_\lambda,\vec{r}} = C_{M_\lambda,\vec{r},\lambda_0,\lambda,\psi_0,\psi}$, 
\begin{equation}
|\psi_j*f(y)| \leq C_{M_\lambda,\vec{r}}\, \sum_{k=j}^\infty  2^{(j-k)(M_\lambda+1)a_0} 
\int \frac{2^{j|\vec{a}|}|\psi_k*f(y-z)|}{\prod_{l=1}^n (1+2^{ja_l}|z_l|)^{r_l}} \, dz.
\label{eq:calderon_beforeSteps}
\end{equation}
Since $j\mapsto 2^{j\vec a\cdot\vec r}\prod_{l=1}^n (1+2^{j a_l}|x_l-z_l|)^{-r_l}$ is 
monotone increasing, \eqref{eq:calderon_beforeSteps} entails that for
$N=(M_\lambda+1)a_0-\vec a\cdot\vec r$,
\begin{gather}
\begin{split}
\psi_j^*f(x) 
&\leq C_{M_\lambda,\vec{r}}\, \sum_{k\geq j}  2^{(j-k)(M_\lambda+1)a_0} 
\int \frac{2^{j|\vec{a}|}|\psi_k*f(z)|}{\prod_{l=1}^n (1+2^{ja_l}|x_l-z_l|)^{r_l}} \, dz\\
&\leq C_{N}\, \sum_{k\geq j} 2^{(j-k)N} 
\int \frac{2^{k|\vec{a}|} |\psi_k*f(z)|^\tau}{\prod_{l=1}^n (1+2^{ka_l}|x_l-z_l|)^{r_l \tau}} \, dz \,\, \left(\psi_k^* f(x)\right)^{1-\tau}.
\label{eq:28inRy}
\end{split}
\end{gather}
Here $N$ can be lowered in the exponent, so \eqref{eq:28inRy} holds for all $N\ge-\vec
a\cdot\vec r$, with $N\mapsto C_{N,\vec r}$ piecewise constant; i.e.\ constant on
intervals having the form $\,](k-1)a_0,ka_0]-\vec a\cdot\vec r$, $k\in\N_0$.

Obviously this yields \eqref{eq:maxFuncEstimate} in case $\tau=1$.

\emph{Step~2.}
To cover a given $\tau\in\, ]0,1[\,$ we apply Lemma~\ref{sum2-lem} with $b_j$ as
the last integral in \eqref{eq:28inRy}: because of the inequality \eqref{eq:28inRy}, the estimate
\eqref{eq:maxFuncEstimate} with $C_{N,\vec r,\tau}=C_N$ follows for all $N>0$ by the lemma if we
can only verify the last assumption that, for some $N_0>0$,
\begin{equation}
  d_j:= \psi_j^*f(x)=O\left(2^{jN_0}\right).
\label{ON0-eq}
\end{equation}
In case $\omega\le \vec r$ for $\omega=\operatorname{ord}_{\cs'} f$, this estimate follows 
for all $j\ge0$ from
standard calculations by applying \eqref{ordS'-id} to the numerator in $\psi_j^*f(x)$.

In the remaining cases, where $\omega>r_l$ for some $l\in\{1,\ldots,n\}$, we shall show a similar
estimate unless \eqref{eq:maxFuncEstimate} is trivial. 
First we choose $\vec{q}$ such that $\vec{q}\ge \max (r_1,\ldots,r_n,\omega)$. Then
\eqref{eq:maxFuncEstimate} holds true for $\vec{q}$ and the right-hand side 
gets larger by replacing each $q_l$ with $r_l$ in the denominator. Hence we have for $N>0$,
\begin{equation}
|\psi_j * f(y)|^\tau \leq C_{N,\vec{q},\tau}\, \sum_{k\ge j} 2^{(j-k)N\tau} 
\int \frac{2^{k|\vec{a}|} |\psi_k*f(z)|^\tau}{\prod_{l=1}^n (1+2^{ka_l}|y_l-z_l|)^{r_l \tau}} \, dz.
\end{equation}
Using monotonicity as in Step~1,
the above is seen to imply, say for $N>\vec a\cdot\vec r$, $j\in\N_0$ that
\begin{equation}
\label{eq:constantWrongParameter}
  \left(\psi_{j}^* f(x)\right)^\tau \leq C_{N,\vec{q},\tau}\, 
  \sum_{k\ge j} 2^{(j-k)(N-\vec{a}\cdot\vec{r}\, )\tau} 
  \int \frac{2^{k|\vec{a}|} |\psi_k*f(z)|^\tau}{\prod_{l=1}^n (1+2^{ka_l}|x_l-z_l|)^{r_l \tau}} \, dz.
\end{equation}
(The constant depends on $\vec{q}$, i.e.\ on $f$.)
We can assume the sum on the right-hand side 
is finite for some $j_1\ge0$, $N_1>\vec a\cdot\vec r$, for else \eqref{eq:maxFuncEstimate} is trivial.
Then 
\begin{gather}
\begin{split}\label{eq:finiteFactorWithMaxFct}
  \sup_{m\ge j_1}\, &2^{(j_1-m)(N_1-\vec{a}\cdot\vec{r}\, )} \psi_{m}^* f(x)\\
  &\leq C_{N_1,\vec{q},\tau}^{1/\tau}\, 
  \Big( \sum_{k\ge j_1} 2^{(j_1-k)(N_1-\vec{a}\cdot\vec{r}\, )\tau} 
  \int \frac{2^{k|\vec{a}|} |\psi_k*f(z)|^\tau}{\prod_l (1+2^{ka_l}|x_l-z_l|)^{r_l\tau}} \, dz
  \Big)^{\frac1{\tau}} <\infty.
\end{split}
\end{gather}
This implies \eqref{ON0-eq} at once for $j\ge j_1$ and $N_0:=N_1-\vec a\cdot\vec r$,
so now Lemma~\ref{sum2-lem} yields \eqref{eq:maxFuncEstimate} for $j\ge j_1$. 
When considering the smallest such $j_1$, the right-hand side of \eqref{eq:maxFuncEstimate}  is infinite for 
every $j<j_1$ (any $N$) so that \eqref{eq:maxFuncEstimate} is trivial.

\emph{Step~3.} For $\tau>1$ we deduce
\eqref{eq:calderon_beforeSteps} with $r_l+1$ for all $l$ and afterwards apply
H\"{o}lder's inequality with dual exponents $\tau,\tau'>1$ with respect to 
Lebesgue measure and the counting measure. Simple calculations then yield \eqref{eq:maxFuncEstimate}.
\end{proof}

Now we can briefly modify the arguments in \cite{Ry99BPT} to obtain the next result.

\begin{thm}\label{inverse}
Let $\psi_0,\psi\in\cs(\Rn)$ satisfy the Tauberian conditions \eqref{con1}--\eqref{con2}.
Whenever $0< \vec{p}<\infty$, $0<q\le \infty$, $-\infty < s < \infty$ and the $\psi_j^*f$ are
defined for $\vec r$ satisfying
\begin{equation}
  r_l \min (q,p_1, \ldots, p_n)>1 , \quad l =1,\dots , n, 
\end{equation}
then there exists a constant $c>0$ such that for $f\in\cs'(\Rn)$,
\begin{equation}\label{eq-inverse}
\| 2^{sj} \psi^*_j f \, | L_{\vec{p}}(\ell_q)\|
\le c\,                    
\| 2^{sj} \psi_j * f\, | L_{\vec{p}}(\ell_q)\|.
\end{equation}
\end{thm}

\begin{proof}
The proof relies on the Hardy-Littlewood maximal function
\begin{equation}
  Mf(x) = \sup_{r>0} \frac{1}{\mbox{meas}(B(0,r))} \int_{B(0,r)} |f(x+y)|\, dy.
\end{equation}
When applied only in one variable $x_l$, we denote it by $M_l$; i.e., using the splitting
$x=(x',x_l,x'')$ we have 
 $ M_l u(x_1,\ldots ,x_n) := (Mu(x',\cdot,x''))(x_l).$
By assumption on $\vec r$, we may pick $\tau$ such that $\max_{1\le
  l\le n}\frac{1}{r_l} < \tau < \min(q,p_1,\ldots,p_n)$.  
This implies that $(1+ |z_l|)^{-r_l\, \tau} \in L_1 (\R)$, and since it is also radially decreasing,
iterated application of the majorant property of the Hardy-Littlewood 
maximal function, described in e.g. \cite[p.~57]{Ste93}, yields a bound 
of the convolution on the right-hand side of \eqref{eq:maxFuncEstimate}, hence
\begin{equation}
  \psi_j^*f(x)\le 
  C_{N,\vec{r}}^{1/\tau} 
  \big(\sum_{k\geq j} 2^{(j-k)N\tau} M_n (\ldots M_2(M_1 |\psi_k * f|^\tau)\ldots )(x)\big)^{1/\tau}.
\end{equation}
Here application of Lemma~\ref{sum1-lem} gives
\begin{alignat}{1}
\| 2^{js}\psi_j^*f\, | L_{\vec{p}}(\ell_q)\|
\leq C_{N,\vec{r}}\, \| 2^{js\tau} M_n( \ldots (M_1 |\psi_j * f|^\tau)\ldots )
 | L_{\vec{p}/\tau}(\ell_{q/\tau}) \|^{1/\tau},
\end{alignat}
whence \eqref{eq-inverse} follows by $n$-fold application of the 
maximal inequality of Bagby~\cite{Bag75} on the space 
$L_{\vec{p}/\tau}(\ell_{q/\tau})$, since $\tau< \min(q,p_1,\ldots,p_n)$;
 cf.\ also \cite[Sec.~3.4]{JoSi08}.
\end{proof}

\section{General quasi-norms and  local means}
\label{localMeans}

First of all Theorems~\ref{maxmod}, \ref{inverse} give some very general
characterisations of $F^{s,\vec a}_{\vec p,q}$. 
In fact the next result shows that in Definition~\ref{F-defn} the
Littlewood--Paley partition of unity is not essential: 
the quasi-norm can be replaced by a more general one in which
the summation to $1$ or the compact supports, or both, are lost:

\begin{thm} \label{general-thm}
Let $s\in\R$, $0<\vec p<\infty$, $0<q\le\infty$ and let 
$\psi_0,\psi\in\cs(\Rn)$ be given such that the Tauberian conditions \eqref{con1}--\eqref{con2}
are fulfilled together with a moment condition of order $M_\psi$
so that $s<(M_\psi+1)\min(a_1,\dots,a_n)$, cf.~\eqref{moment}.
When $\psi_{j,\vec a}^*f$ is defined with $\vec r>\min(q,p_1,\dots,p_n)^{-1}$,
cf.~\eqref{eq:maximalPFStype}, then the following properties of $f\in{\mathcal{S}}'({\mathbb{R}}^n)$ are equivalent:
  \begin{alignat*}{1}
    \text{\upn{(i)}}\quad & f \in F^{s,\vec a}_{\vec p,q}(\Rn),
\\
   \text{\upn{(ii)}}\quad & \| \{ 2^{sj}\psi_j*f\}_{j=0}^{\infty} | L_{\vec p}(\ell_q) \| <\infty,
\\
  \text{\upn{(iii)}}\quad &  \| \{ 2^{sj}\psi_{j,\vec a}^*f\}_{j=0}^{\infty} | L_{\vec p}(\ell_q)
  \| < \infty.
  \end{alignat*}
Moreover, the $F^{s,\vec a}_{\vec p,q}$-quasi-norm is equivalent to
those in \upn{(ii)} and \upn{(iii)}.
\end{thm}
\begin{proof}
  Since $\psi_j*f(x)\le \psi_{j,\vec a}^*f(x)$ is trivial, clearly \upn{(iii)}$\implies$\upn{(ii)};
  the converse holds by Theorem~\ref{inverse}. To obtain
  \upn{(iii)}$\implies$\upn{(i)}, one may in the Lizorkin--Triebel norm estimate the
  convolutions by $(\cfi \Phi)_{j,\vec a}^*f$, and the resulting norm is estimated by the one in \upn{(iii)}
  by means of Theorem~\ref{maxmod} (with a trivial index set like $\Theta=\{1\}$).
  That \upn{(i)}$\implies$\upn{(iii)} follows by using Theorem~\ref{maxmod} to estimate from
  above by the quasi-norm defined from $(\cfi \Phi)_{j,\vec a}^*f$, with all $r_l$ so large that
  Theorem~\ref{inverse} gives control by the $\cfi \Phi_j*f$.
\end{proof}

From the above it is e.g.\ obvious that the space $F^{s,\vec a}_{\vec p,q}$ does not
depend on the Littlewood--Paley partition of unity in \eqref{unity}, and that different choices yield 
equivalent quasi-norms.

As an immediate corollary of Theorem~\ref{general-thm},
there is the following characterisation of $F^{s,\vec a}_{\vec p,q}$
in terms of integration kernels.
It has been well known in the isotropic case:

\begin{thm}\label{local}
Let $k_0,k^0 \in \cs(\Rn)$ such that $\int k_0(x)\, dx \neq 0 \neq \int k^0(x)\, dx$ and
set $k (x)= \Delta^N k^0(x)$ for some $N\in\N$.
When $0 < \vec{p} <\infty$, $0< q \le \infty$, and $s < 2N\, \min(a_1,\ldots,a_n)$,
then a distribution $f \in \cs'(\Rn)$ belongs to $F^{s,\vec{a}}_{\vec{p},q}(\Rn)$ 
if and only if 
\begin{equation}\label{eq:charOfF}
\|\, f \, | F^{s,\vec{a}}_{\vec{p},q}\|^* :=
\| \, k_0 * f\, |L_{\vec{p}}\| + \| \{2^{sj} k_j * f\}_{j=1}^\infty \, | L_{\vec{p}}(\ell_q)\|<\infty .
\end{equation}
Furthermore, $\|\, f \, | F^{s,\vec{a}}_{\vec{p},q}\|^*$ is an equivalent
quasi-norm on $F^{s,\vec{a}}_{\vec{p},q}(\Rn)$.
\end{thm}

In \eqref{eq:charOfF}, the functions $k_j$, $j\ge1$ are 
given by $k_j(x)=2^{j|\vec a|}k(2^{j\vec a}x)$; cf.~\eqref{eq:defOfSubscript_j}.

\begin{rem}
Obviously, we may choose $k_0,k^0$ such that both functions have compact support.
In this case Triebel termed $k_0$ and $k$ kernels of \emph{local means},
and in \cite[2.4.6]{T3} he proved that \eqref{eq:charOfF} is an equivalent
quasi-norm on the $f$ belonging a priori to the isotropic space $F^{s}_{p,q}$. 
This was carried over to anisotropic, but unmixed spaces by Farkas \cite{Far00}.
Extension to function spaces with generalised smoothness 
has been done by Farkas and Leopold \cite{FaLe06}; 
and to spaces of dominating mixed smoothness by Vybiral~\cite{Vyb06} and Hansen~\cite{Han10}.
\end{rem}

\begin{rem}
Bui, Paluszinki and Taibleson \cite{BuPaTa96} obtained a characterisation, 
i.e.\ equivalence for all 
$f\in\cs'$,  in the isotropic (but weighted) case,  
which Rychkov~\cite{Ry99BPT} simplified to the present discrete Littlewood-Paley decompositions.
Our Theorem~\ref{local} generalises this in two ways, i.e.\ we prove a
characterisation of $F^{s,\vec a}_{\vec p,q}$ that has anisotropies both in terms of $\vec a$ and
mixed norms.
\end{rem}

\providecommand{\bysame}{\leavevmode\hbox to3em{\hrulefill}\thinspace}
\providecommand{\MR}{\relax\ifhmode\unskip\space\fi MR }
\providecommand{\MRhref}[2]{%
  \href{http://www.ams.org/mathscinet-getitem?mr=#1}{#2}
}
\providecommand{\href}[2]{#2}

\end{document}